\documentclass[12pt]{article}
\usepackage{amsmath, amsthm, amssymb}
\usepackage[pdftex]{graphicx}
\usepackage{color}

\begin{document}
\newtheorem{theorem}{Theorem}
\newtheorem{conjecture}[theorem]{Conjecture}
\newtheorem{lemma}[theorem]{Lemma}
\newtheorem{claim}[theorem]{Claim}
\newtheorem{problem}[theorem]{Problem}

\theoremstyle{definition}
\newtheorem{definition}[theorem]{Definition}

\centerline{{\Large \bf Odd $K_{3,3}$ subdivisions in bipartite graphs}%
\footnote{Partially supported by NSF under Grant No.~DMS-1202640. 13 September 2013}}

\bigskip
\bigskip

\centerline{{\bf Robin Thomas}%
}
\smallskip
\centerline{and}
\smallskip
\centerline{{\bf Peter Whalen}}
\bigskip
\centerline{School of Mathematics}
\centerline{Georgia Institute of Technology}
\centerline{Atlanta, Georgia  30332-0160, USA}
\bigskip

\begin{abstract}
\noindent
We prove that every internally $4$-connected non-planar bipartite graph
has an odd $K_{3,3}$ subdivision; that is, a subgraph obtained from $K_{3,3}$
by replacing its edges by internally disjoint odd paths with the same ends.
The proof gives rise to a polynomial-time algorithm to find such a subdivision.
(A bipartite graph $G$ is {\em internally $4$-connected} if it is $3$-connected, has 
at least five vertices, and there is no partition $(A,B,C)$ of $V(G)$ such that
$|A|,|B|\ge2$, $|C|=3$ and $G$ has
no edge with one end in $A$ and the other in $B$.)
\end{abstract}

\section{Introduction}\label{intro}

All graphs in this paper are finite and simple.
Kuratowski's theorem~\cite{Kur} gives a characterization of planar graphs as
those graphs that have no subgraph isomorphic  to a subdivision of $K_5$ or $K_{3,3}$.
Both $K_5$ and $K_{3,3}$ are necessary in the statement, but one can
argue that $K_{3,3}$ is the more important of the two.
It is easy to reduce planarity  testing to $3$-connected graphs, and for
$3$-connected graphs subdivisions of $K_5$ are not needed, in the following
sense.

\begin{theorem}
\label{kuratowskivar}
A $3$-connected graph is not planar if and only if either it is isomorphic to $K_5$
or it has a subgraph isomorphic to a subdivision of  $K_{3,3}$.
\end{theorem}

\noindent
Theorem~\ref{kuratowskivar} is well-known and can easily be derived from
Kuratowski's theorem. 
Since we will be concerned with subdivisions of $K_{3,3}$, we make the following
definition.

\begin{definition}
Let $G$ be a graph and $H$ a subgraph of $G$ isomorphic to a subdivision of $K_{3,3}$.  
Let $v_1, v_2, \ldots, v_6$ be the degree three vertices of $H$ and for $i=1,2,3$ and $j=4,5,6$ let
$P_{ij}$ be the paths in $H$ between $v_i$ and $v_j$.  
We then refer to $H$ as a \emph{hex} or a \emph{hex of $G$}, the vertices $v_i$ as the \emph{feet} of $H$, 
and the paths $P_{ij}$ as the \emph{segments} of $H$.
A segment is {\em odd} if it has an odd number of edges, and {\em even} otherwise.
A hex $H$ is {\em odd} if every segment of $H$ is odd.
\end{definition}

While working on Pfaffian orientations (described later) we were led to the following
variation of Theorem~\ref{kuratowskivar}. If $G$ is $3$-connected, bipartite and non-planar, must it have an odd hex?
Unfortunately, that is not true, as the graph depicted in Figure~\ref{fig:counter} shows, but it is true if we
increase the connectivity slightly.
We say that a bipartite graph $G$ is {\em internally $4$-connected} if it is $3$-connected, has 
at least five vertices, and there is no partition $(A,B,C)$ of $V(G)$ such that
$|A|,|B|\ge2$, $|C|=3$ and $G$ has
no edge with one end in $A$ and the other in $B$.
The following is our main result.

\begin{figure*}[ht]
\begin{center}
\includegraphics[scale=1]{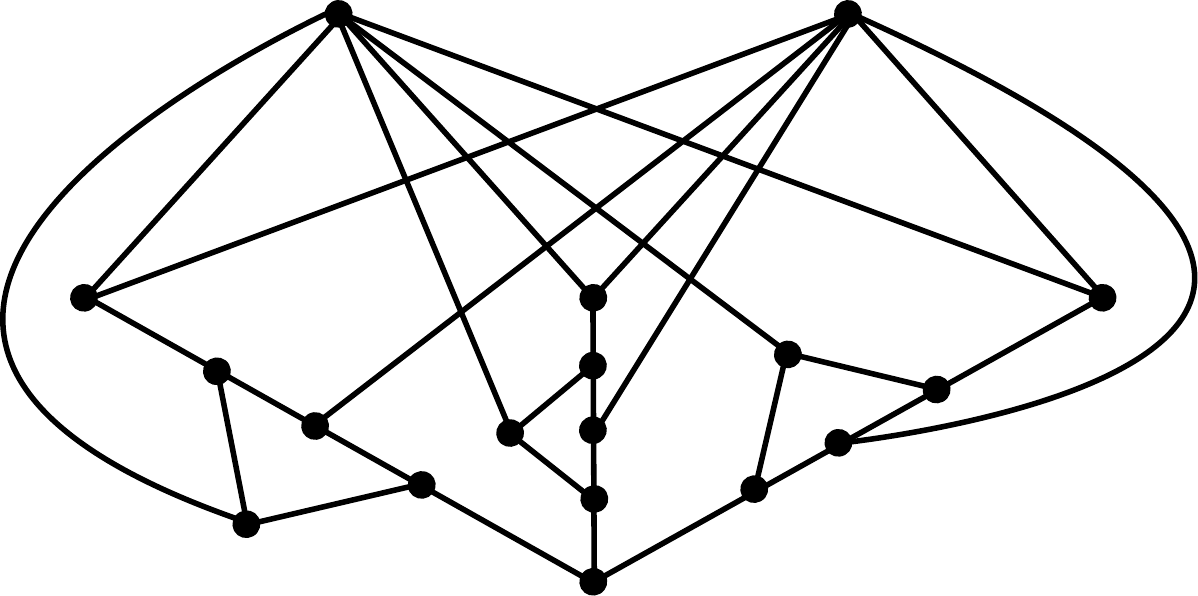}
\caption{A graph showing that Theorem~\ref{MainThm} does not extend to $3$-connected graphs}
\label{fig:counter}
\end{center}
\end{figure*}

\begin{theorem} 
\label{MainThm}
Every internally $4$-connected bipartite non-planar graph has an odd hex.  
\end{theorem}

Let us now explain the notion of a Pfaffian orientation, and how it led us to the above theorem.
Let $H$ be a subgraph of a graph $G$. We say that $H$ is a {\em central subgraph} if $G\backslash V(H)$
has a perfect matching.
(We use $\backslash$ for deletion and $-$ for set-theoretic difference.)
An orientation $D$ of a graph $G$ is called {\em Pfaffian} if every even central cycle has an odd number of
edges directed in either direction of the cycle.
A graph is called {\em Pfaffian} if it admits a Pfaffian orientation. 
Pfaffian orientations have been introduced by Kasteleyn~\cite{Kas,Kast2,Kast3},
who demonstrated
that one can enumerate perfect matchings in a Pfaffian graph in polynomial time.
That is significant, because counting the number of perfect matchings is $\#$P-complete \cite{Valiant} in general graphs.
It is not known whether there is a polynomial-time algorithm to test if a graph is Pfaffian,
but we shall see below that there is one for bipartite graphs.
The latter is noteworthy, because it implies polynomial-time algorithms for other problems
of interest, such as P\'olya's permanent problem, the even directed cycle problem,
the sign non-singular matrix problem, and others.
A survey of Pfaffian orientations may be found in~\cite{ThoPfafSurv}.
The following is a result of Little \cite{Little}.

\begin{theorem}
\label{thm:little}
A bipartite graph is Pfaffian if and only if it does not have an odd hex as a central subgraph.
\end{theorem}

\noindent
While Little's theorem is elegant, it is not clear how to use it to decide in polynomial time whether a bipartite graph is Pfaffian.
A polynomial-time decision algorithm was obtained in~\cite{RST} using a different method---it is based on a structure theorem
proved in~\cite{RST} and independently in~\cite{McCPolya}.
However, both proofs of the structure theorem are fairly long.

We were wondering whether a simpler-to-justify algorithm may be obtained using Theorem~\ref{thm:little}, as follows.
First a definition.
A bipartite graph $G$ is a {\em brace} if $G$ is connected, has at least five vertices
 and every matching of size at most two is a subset
of a perfect matching.
It is easy to see that the decision problem whether a bipartite graph is Pfaffian can be reduced to braces,
and that every brace is internally $4$-connected.
So let $G$ be a brace.
By Theorem~\ref{thm:little} we want to test whether $G$ has an odd hex as a central subgraph.
To that end we may assume that $G$ has a hex, for otherwise $G$ is Pfaffian.
In fact, we can find a hex in linear time using one of the linear-time planarity algorithms.
The next step is to decide whether $G$ has an odd hex.
It follows from~\cite[Theorem~(1.5)]{RST} that
if a brace has a hex, then it has an odd hex,
but it occurred to us that this should be true more generally than for braces,
and that is how we were led to Theorem~\ref{MainThm}.
The next step in our program is, given an odd hex in a brace $G$, decide whether there is an odd hex that is
a central subgraph.
We were able to do that, and will report on it elsewhere.

The paper is organized as follows.
In the next section we prove two lemmas, and in Section~\ref{sec:proof} we prove Theorem~\ref{MainThm}.
We start with an arbitrary hex, and gradually increase the number of odd segments in it.



\section{Lemmas}

In this section we prove two lemmas that we will need for the proof of Theorem~\ref{MainThm}.

\begin{lemma} \label{mainlemma}
Let $G$ be an internally 4-connected bipartite graph with bipartition $(A, B)$.  Let $a, v \in A$ and $b, c \in B$ with paths 
$P_1 = v \ldots a, P_2 = v \ldots b, P_3 = v \ldots c$ vertex disjoint except for $v$.  Let $X \subseteq V(G)$ with $|X| \geq 2$ be disjoint from $V(P_1) \cup V(P_2) \cup V(P_3)$.  Then at least one of the following holds:
\begin{enumerate}
\item[{\rm(1)}]{There exist $v' \in A, u \in B, x \in X$ and paths 
$P_1' = v' \ldots a, P_2' = v' \ldots b, P_3' = v' \ldots c, P_4' = u \ldots x$ such that $u \in V(P_1')$ and all of the $P_i'$ are vertex disjoint and are disjoint from $X$ except that $v' \in V(P_1') \cap V(P_2') \cap V(P_3'), u \in V(P_1') \cap V(P_4')$ and $x \in X \cap V(P_4')$.}
\item[{\rm(2)}]{There exists $v', s \in A, u, t \in B, x \in X$ and paths $P_1' = v' \ldots a, P_2' = v' \ldots t \ldots s \ldots b, P_3' = v' \ldots c, P_4' = u \ldots s, P_5' = t \ldots x$ such that $u \in V(P_1')$ and all of the $P_i$ are vertex disjoint and are disjoint from $X$ except that $v' \in V(P_1') \cap V(P_2') \cap V(P_3'), u \in V(P_1') \cap V(P_4'), s \in V(P_2') \cap V(P_4'), t \in V(P_2') \cap V(P_5'), x \in V(P_5') \cap X$.}
\end{enumerate}
\end{lemma}

\begin{definition}
We will refer to the paths $P_1', P_2',$ and $P_3'$ as the \emph{replacement paths} and the paths 
$P_4'$ and, when appropriate, $P_5'$ as the \emph{new paths}.  
In the forthcoming arguments we will apply either the induction hypothesis or Lemma~\ref{mainlemma}
to various carefully selected paths $R_1, R_2, R_3$ to obtain replacement paths $R_1', R_2', R_3'$ and new paths $R_4'$ and, when appropriate, $R_5'$.  
However, we will be able to assume that $R_1' = R_1, R_2' = R_2$ and $R_3' = R_3$ which will simplify our notation.  
We will refer to this assumption as assuming that the \emph{replacement paths do not change}.
\end{definition}

\begin{proof}[Proof of Lemma~\ref{mainlemma}]
Let the paths $P_1,P_2,P_3$ be fixed.
By an \emph{augmenting sequence} we mean a sequence of paths $Q_1, \ldots, Q_k$, 
where the ends of $Q_i$ are $v_{2i-1}$ and $v_{2i}$, $v_{2k} \in X, v_1 \in V(P_1) - \{a, v\}$, each other $v_i$ is
in $P_j \backslash v$ for some $j \in \{1, 2, 3\}$, and all the $Q_i$ are vertex disjoint from one another and disjoint from the $P_i$ and $X$ except for their ends.  Further, for $j > 1$ odd, $v_j$ and $v_{j-1}$ are distinct and both lie on the same $P_i$ and $v_j$ lies between $v$ and $v_{j-1}$ on $P_i$.  
If $v_i, v_j \in V(P_l)$ and $i < j-1$, then $v, v_i, v_j$ appear in $P_l$ in the order listed (possibly $v_i=v_j$).  
We refer to each $Q_i$ as an \emph{augmentation}.  
The \emph{length} of an augmenting sequence is the number of augmentations it has. 
We define the \emph{index} of the augmenting sequence $Q_1,Q_2,\ldots,Q_k$ to be the smallest integer $i$ such that
either  $i$ is odd and $v_i\in A$, or $i$ is even and $v_i\in B$, or $i=2k + 1$.

We proceed by induction on the size of $V(G) - X$. 
Since $G$ is internally 4-connected it follows by the standard ``augmenting path" argument from network flow
theory or from Lemmas~3.3.2 and~3.3.3 in~\cite{DieGT} that there exists an augmenting sequence.


Choose the vertex $v$, paths $P_1, P_2, P_3$, and an augmenting sequence $S = (Q_1, \ldots, Q_k)$
such that the length of $S$ is as small as possible, and, subject to that, the index of $S$ is as large as possible.  
Let $v_1,v_2,\ldots,v_{2k}$ be the ends of the paths $Q_i$, numbered as above.
Then it follows that for $j=2,4,\ldots,2k-2$ the vertex $v_j$ lies on a different path $P_i$ than $v_{j-1}$. 
Note that this lemma is equivalent to showing that the length of $S$ is at most $2$ and that the index of $S$ is at least twice the length of $S$.

Suppose first that the length of $S$ is $1$.  Then we may assume that the index of $S$ is $1$, so $v_1 \in A$.  
Let $X' = X \cup V(v_1P_1a \cup Q_1 \backslash v_1)$.  
Then apply the induction hypothesis to the paths $vP_1v_1, P_2, P_3$ and set $X'$.  
We may assume that the replacement paths do not change.
Suppose we have outcome (1).
Thus there exists a path $P_4$ with ends $u\in B\cap V(vP_1v_1)$ and $x\in X'$, disjoint from $V(P_1\cup P_2\cup P_3)\cup X'$,
except for its ends.
If $x \in X$, then this is exactly outcome (1) in the original situation.  
If $x \in V(Q_1)$, then take $P_4' = P_4 \cup xQ_1v_2$ to find outcome (1) in the original situation.  
If $x \in V(v_1P_1a) - \{v_1\}$, then take $P_1' = vP_1u \cup P_4' \cup xP_1a$, 
$P_4' = uP_1v_1 \cup Q_1$ to again have outcome (1).  
So we must have outcome (2), and so there exist paths $P_4,P_5$ as stated in (2).  
Again, if $x \in X$, this is exactly outcome (2), and if $x \in V(Q_1)$, then taking 
$P_5' = P_5 \cup xQ_1v_2$ again gives outcome (2).  
So we may assume $x \in V(v_1P_1a)$.  
We then take $v' = v$, $P_1' = vP_2tP_5xP_1a$, $P_2' = vP_1uP_4sP_2b$, $P_3' = vP_3c, P_4' = tP_2s, P_5' = uP_1v_1Q_1v_2$,
which is an instance of outcome (2).

So we may assume that the length of $S$ is at least $2$.  Suppose that the index of $S$ is $1$, so $v_1 \in A$.  Without loss of generality, we may assume that $v_2 \in V(P_2)$.  Let $u \in B$ lie on $P_1'$ between $v_1$ and $a$.  Since $\{a, v_1\}$ is not a $2$-separation in $G$, we can apply Menger's Theorem to find three paths from $u$, one to $a$, one to $v_1$ and one to $V(X) \cup V(Q_1) \cup V(Q_2) \cup V(vP_1v_1) \cup V(P_2) \cup V(P_3)$ labeled $R_1, R_2, R_3$ respectively.  
We replace $v_1P_1a$ by $R_1 \cup R_2$ and simply refer to $R_3$ as $R$.  Let the ends of $R$ be $u$ and $r$.  
If $r \in X$, then we have an augmenting sequence of length $1$ contrary to the choice of $S$.  
If $r \in V(Q_i)$, then we have found an augmenting sequence of at most the same length as $S$, but with index at least $2$.  
If $r \in V(P_1)$, then we take $P_1' = vP_1rRuP_1a$ and $Q_1' = uP_1v_1Q_1v_2$, 
which gives an augmenting sequence of the same length but with higher index.  
If $r$ is on $P_2$ between $v_3$ and $b$ with $r \neq v_3$, then taking $Q_1' = R$ is immediately an augmenting sequence with the same length and higher index.  If $r$ is on $P_2'$ between $v_3$ and $v'$, then let $v' = v_1$, $P_1' = v_1P_1a$, $P_2' = v_1Q_1v_2P_2b$, $P_3' = v_1P_1vP_3c$, and $Q_1 = uRrP_2v_3Q_2v_4$ which gives an augmenting sequence of shorter length.  Similarly, if $r$ is on $P_3'$, let $v' = v_1$, $P_1' = v_1P_1a$, $P_2' = v_1Q_1v_2P_2b$, $P_3' = v_1P_1vP_3c$, and $Q_1' = R$, $Q_2' = vP_2v_3Q_2v_4$ which is an augmenting sequence of the same length but higher index.

So we may assume that the index of $S$ is at least $2$.  Suppose the index is exactly $2$, so $v_2 \in B$.  
Then apply the induction hypothesis to the paths $Q_1$, $v_1P_1a$ and $v_1P_1v$ and set
$X': = X \cup V(P_2) \cup V(P_3) \cup V(Q_2) \cup V(Q_3) \cup \cdots \cup V(Q_k) - \{v, v_2\}$.  
We assume, since we may, that the replacement paths do not change.  Suppose we have outcome (1).  
This gives a vertex $u \in A$ on $Q_1$ and $x$ in $X'$ with a path $P_4$ between them.  
If $x \in X$, then take $u' = v_1$, $P_4' = v_1Q_1uP_4x$ to get outcome (1).  If $x \in Q_i$ for $i > i$, then take $Q_1' = v_1Q_1uP_4xQ_i$ which gives the shorter augmenting sequence $Q_1', Q_{i+1}, Q_{i+2}, \ldots, Q_k$.  If $x$ is on $P_2$ between $v_2$ and $b$, then take $P_2' = vP_1v_2Q_1uP_4xP_2b$, $v_2' = u$, and $Q_1' = v_1Q_1u$ which gives an augmenting sequence of higher index.  If $x$ is on $P_2$ between $v$ and $v_2$, then we take $P_2' = vP_2xP_4uQ_1v_2P_2b$, $v_3' = x$, and $Q_2' = xP_2v_3Q_2v_4$ if $v_3 \in v_2P_2x$ and $Q_2' = Q_2$ if $v_3 \in xP_2v$ which gives an augmenting sequence with higher index.  Finally, if $x$ is on $P_3$, then take $v' = u, P_1' = uQ_1v_1P_1a, P_2' = uQ_1v_2P_2b, P_3' = uP_4xP_3c$, and $Q_1' = v_1P_1vP_2v_3Q_2v_4$ which gives a shorter augmenting sequence.

So instead we have outcome (2) and we use the notation for $u, s, t, x, P_4, P_5$ as listed in the outcome.  
It follows that $P_1 = vP_1sP_1tP_1v_1P_1a$.  If $x \in X$, then taking $P_1' = vP_1sP_4uQ_1v_1P_1a, P_2' = P_2, P_3' = P_3, P_4' = sP_1tP_5x$ gives outcome (1).  
If $x \in V(P_2')$ between $v_2$ and $b$, we take $v' = t, P_1' = tP_1a, P_2' = tP_5xP_2b, P_3' = tP_1vP_3c, Q_1' = v_1Q_1v_2P_2v_3Q_2v_4$ and then $Q_1', Q_3, Q_4, ..., Q_k$ is an augmenting sequence of length $k-1$, a contradiction.  If $x$ is on $P_2'$ between $v$ and $v_2$, take $v' = u, P_1' = uQ_1v_1P_1a, P_2' = uQ_1v_2P_2b, P_3' = uP_4sP_1vP_3c, Q_1' = v_1P_1tP_5xP_2v_3Q_2v_4$ which gives a shorter augmenting sequence.  If $x \in V(Q_i)$ with $i > 1$, then taking $Q_1' = tP_5xQ_iv_{2i}$ gives the augmenting sequence $Q_1', Q_{i+1}, Q_{i+2}, ..., Q_k$ which contradicts the choice of $S$.  Finally, if $x$ is on $P_3'$, take $v' = t, v_2' = u, v_3' = s, P_1' = tP_1a, P_2' = tP_1sP_4uQ_1v_2P_2b, P_3' = tP_5xP_3c, Q_1' = v_1Q_1u$, and $Q_2' = sP_1vP_2v_3Q_2$ to get an augmenting sequence of the same length and higher index.

So we may assume the index of $S$ is at least $3$.  Suppose the index is exactly $3$, so $v_3 \in A$.  Note that $v_2$ and $v'$ are completely symmetric with respect to this augmenting sequence (up to $v_3$).  We apply induction to the paths $v_2P_2v_3, Q_1, v_2P_2b$ and set $X' := X \cup V(P_1 \cup vP_2v_3 \cup P_3 \cup Q_2 \cup Q_3 \cup \cdots \cup Q_k) - \{v_1, v_3\}$.  Since we may, we assume the replacement paths do not change.  Suppose first we find outcome (2).  
We use the notation in the outcome for $u, s, t, x$.  If $x \in X$, then taking $P_4' = v_1Q_1tP_5x$ with $u' = v_1$ gives outcome (1).  
If $x$ is on $Q_i, i > 1$, we take $Q_1' = v_1Q_1tP_5xQ_i$ which gives a shorter augmenting sequence.  If $x$ is on $P_2$ between $v$ and $v_3$, $P_3$, or on $P_1$ between $v_1$ and $v$, let $i$ be such that $P_i$ contains $x$, then we can take $v' = v_2, P_1' = v_2P_2uP_4sQ_1v_1P_1a, P_2' = v_2P_2b, P_3' = v_2Q_1tP_5xP_ivP_3c$, $v_1' = u, Q_1' = uP_2v_3Q_2v_4$ to find a shorter augmenting sequence.  Finally, if $x$ is on $P_1$ between $v_1$ and $a$, we take $v' = v_2$, $P_1' = v_2Q_1tP_5xP_1a, P_2' = v_2P_2b, P_3' = v_2P_2uP_4sQ_1v_1P_1vP_3c, v_1' = t, v_2' = s, v_3' = u, Q_1' = tQ_1s, Q_2' = uP_2v_3Q_2v_4$ which gives an augmenting sequence with the same length and higher index.

So instead we consider outcome (1) with the notation for $u, x, P_4$ as in the outcome.  
If $x \in X$, we can take $u' = v_1, s' = v_2, t' = u, P_4' = Q_1, P_5' = P_4$ to find outcome (2).  
If $x$ is on a $Q_i, i > 1$, then we take $v_3' = u, Q_2' = uP_4xQ_i$ which gives an augmenting sequence of at most the same length but higher index.  If $x$ is on $P_2$ between $v'$ and $v_3$, then take $P_2' = vP_2xP_4uP_2b$, $v_3' = u, Q_2' = uP_2v_3Q_2v_4$ to find an augmenting sequence of the same length with higher index.  If $x$ is on $P_3$, we take $v' = v_2, P_1' = v_2Q_1v_1P_1a, P_2' = v_2P_2b, P_3' = v_2P_2uP_4xP_3c, Q_1' = v_1P_2vP_2v_3Q_2v_4$ which gives a shorter augmenting sequence.  If $x$ is on $P_1$ between $v_1$ and $a$, then take $v' = v_2, P_1' = v_2P_2uP_4xP_1a, P_2' = v_2P_2b, P_3' = v_2Q_1v_1P_1vP_3c, v_1' = u, Q_1' = uP_2v_3Q_2v_4$ which gives a shorter augmenting sequence.  Finally, suppose $x$ is on $P_1$ between $v$ and $v_1$ and $x \in A$.  Then consider $P_1' = v_2Q_1v_1P_1a, P_2' = v_2P_2b, P_3' = v_2P_2uP_4xP_1vP_3c, Q_1' = v_1P_1x, Q_2' = uP_2v_3Q_2$ which gives an augmenting sequence of the same length and lower index.

So we must have $x \in B$ on $P_1$ between $v$ and $v_1$.  We apply induction to the paths $vP_2v_3, vP_1x, P_3$ and set $X' := X \cup V(xP_1a \cup v_3P_2b \cup P_4 \cup Q_1 \cup Q_2 \cup Q_3 \cup \cdots \cup Q_k) - \{x, v_3\}$.  This gives us $u_2 \in A$ between $v$ and $v_3$ with a path $P_5$ to $x_2$.  Suppose we have outcome (1).  
Then if $x_2$ is not on $Q_1$, $P_4$, or $xP_1v_1$, then by symmetry we can apply the analysis of the previous paragraph.  If $x_2$ is on $P_4$, then we replace $P_4$ with $uP_4x_2P_5u_2$ and have one of the outcomes above.  If $x_2$ is on $Q_1$, then take $v' = v, P_1' = P_1, P_2' = vP_2u_2P_5x_2Q_1v_2P_2b, P_3' = P_3, v_1' = x, Q_1' = xP_4uP_2v_3Q_2v_4$ to find a shorter augmenting sequence.  Finally, if $x_2$ is on $xP_1v_1$, then take $v' = v_2, P_1' = v_2Q_1v_1P_1a, P_2' = v_2P_2b, P_3' = v_2P_2uP_4xP_1vP_3c, Q_1' = v_1P_1x_2P_5u_2P_2v_3Q_2$ which is a shorter augmenting sequence.  So we must have outcome (2) of the lemma which gives vertices $u_2, s, t, x_2$ and paths $P_5$ and $P_6$.  If $x_2$ is not on either $P_4$ or $v_1P_1x$, then we can apply the analysis from the previous two paragraphs.  Suppose $x_2 \in V(P_4)$.  Then take $v' = v, P_1' = vP_2u_2P_5sP_1a, P_2' = vP_1tP_6x_2P_4uP_2b, P_3' = P_3, Q_1' = u_2P_2v_3Q_2$ to get a shorter augmenting sequence.  Finally, suppose $x_2 \in V(v_1P_1x)$.  Then take $v' = s, P_1' = sP_1tP_6x_2P_1a, P_2' = sP_1xP_4uP_2b, P_3' = sP_5u_2P_2vP_3c, Q_1' = Q, Q_2' = uP_2v_3Q_2$ which is an augmenting sequence of the same length and lower index.

So we may finally assume that the length of $S$ is at least 3 with index at least $4$.  Note that $v_4$ must be on $P_3'$ (still assuming that $v_2$ was on $P_2'$), since otherwise we get a shorter augmenting sequence.  But then take $v' = v_2, P_1' = v_2Q_1v_1P_1a, P_2' = v_2P_2b, P_3' = v_2P_2v_3Q_2v_4P_3c, Q_1' = v_1P_1vP_3v_5Q_3v_6$, which gives a shorter augmenting sequence.
\end{proof}

Lemma~\ref{mainlemma} will suffice for most of our arguments.
However, on one occasion we will need the following strengthening.

\begin{lemma} \label{lemma2}
Let $G$ be an internally 4-connected bipartite graph with bipartition $(A, B)$.  Let $a, v \in A$ and $b, c \in B$ with paths $P_1 = v ... a, P_2 = v ... b, P_3 = v ... c$ vertex disjoint except for $v$.  Let $X \subseteq V(G)$ be disjoint from $V(P_1) \cup V(P_2) \cup V(P_3)$.  Then at least one of the following holds:
\begin{enumerate}
\item[{\rm(A)}]There exist vertices $v' \in A, u \in B, x \in X \cap A$ and paths $P_1' = v' ... a, P_2' = v' ... b, P_3' = v' ... c, P_4' = u ... x$ such that $u \in V(P_1')$ and all of the $P_i'$ are vertex disjoint and are disjoint from $X$ except as specified,
\item[{\rm(B)}] There exist vertices $v', s \in A, u, t \in B, x \in X \cap B$ and paths 
$P_1' = v' ... a, P_2' = v' ... b, P_3' = v' ... c, P_4' = u ... s ... x, P_5' = s ... t$ 
such that $u \in V(P_1')$, $t \in V(X \cup P_2' \cup P_3')$, and all of the $P_i'$ are vertex 
disjoint and are disjoint from $X$ except as specified and except that $t$ may lie on $P_2'$ or $P_3'$,
\item[{\rm(C)}] There exist vertices $v', s \in A, u, t \in B, x \in X$ and paths 
$P_1' = v' ... a$, $P_2' = v' ... t ... s ... b$, $P_3' = v' ... c$, $P_4' = u ... s$, $P_5' = t ... x$ 
such that $u \in V(P_1')$ and all of the $P_i'$ are vertex disjoint and are disjoint from $X$ except as specified,
\item[{\rm(D)}]  There exist vertices $v', s, w \in A, u, t \in B, x, y \in X \cap B$ and paths $P_1' = v' ... u ... w ... t ... a$,
$P_2' = v' ... b, P_3' = v' ... c, P_4' = u ... s ... x, P_5' = s ... t, P_6' = w ... y$ such that all of the $P_i'$ are vertex disjoint and are disjoint from $X$ except as specified and except that $x$ may equal $y$.
\end{enumerate}
\end{lemma}
\begin{proof}
We proceed by induction on the size of $|V(G)| - |X|$.  Apply Lemma \ref{mainlemma}.  We may assume that we are in the first outcome of that lemma since the second outcome is outcome (C) of this lemma.  
We may assume that the replacement paths do not change.
Thus there exists a path $P_4$ with ends $u\in B\cap V(P_1)$ and $x\in X$, vertex-disjoint from $P_1\cup P_2\cup P_3$,
except for $u$.
We may assume that $x\in B$, for otherwise (A) holds.


We apply the induction hypothesis to the paths $P_4, uP_1a$ and $uP_1v$, and set $X' = X \cup V(P_2 \cup P_3 \setminus \{x, v\})$.  
Note that $|X'| > |X|$ since $b$ and $c$ are distinct, not in $X$, and in $P_2 \cup P_3$ and $v$ was not in $X$ originally.  
We consider each of the four outcomes separately. We may assume that the replacement paths do not change.

If outcome (A) holds, then we obtain outcome (B) of the lemma.
Next, let
%
%
 us assume that the induction hypothesis yields outcome (B).
Thus there exist vertices $s\in A\cap V(P_4)$, $t\in B$, $y\in A\cap (V(P_2\cup P_3)\cup X)$ 
and $w\in A\cap (V(P_1\cup P_2\cup P_3)\cup X)$,
and paths $P_5$ from $s$ to $y$
and $P_6$ from $t$ to $w$ such that the paths $P_5$ and $P_6$ are disjoint and disjoint
from $V(P_1\cup P_2\cup P_3)\cup X$, except as stated.
If $y \in X$, then we have the outcome (A) by taking $P_1' = P_1, P_2' = P_2, P_3' = P_3, P_4' = uP_4sP_5y$ 
and $v' = v, u = u, x = y$.  So without loss of generality, we may assume $y \in V(P_2)$.  
We are now interested in where $w$ lies.  
If $w$ lies on $P_2$, then we may assume that it lies between $y$ and $b$ since $y$ and $w$ are then symmetric.  
In that case, take $P_1' = P_1, P_2' = vP_2yP_5tP_6wP_2b, P_3' = P_3, P_4' = P_4, P_5' = sP_5t, 
x = x, t = t, s = s, u = u, v' = v$ which is exactly outcome (B).  
If $w$ lies on $P_1$ between $v$ and $u$, take 
$v' = w, s = s, t = t, u = u, x = x, P_1 = wP_1a, P_2 = wP_6tP_5yP_2b, P_3 = wP_1vP_3c, P_4 = uP_4x, P_5 = sP_5t$ 
which is again outcome (B).  
If $w$ lies on $P_3$, , take $v' = w, s = s, t = t, u = u, x = x, 
P_1 = wP_3vP_1a, P_2 = wP_6tP_5yP_2b, P_3 = wP_3c, P_4 = uP_4x, P_5 = sP_5t$, which is again outcome (B).  
So we have that $w$ lies on $P_1$ between $u$ and $a$ (or is $a$).

Let $r \in B$ lie between $v$ and $y$ on $P_2$.  
By Menger's theorem and by replacing $vP_2y$ if necessary we may assume that there exists a path $P_7$
from $r$ to a vertex $z$ not on $vP_2y$ that is disjoint from $P_1\cup\cdots \cup P_6$, except for  its  ends.
If $z \in X$, then keeping $v = v', P_1' = P_1, P_2' = P_2, P_3' = P_3$ and taking $u = u, t = r, s = y, x = z, P_4' = uP_4sP_5y, P_5' = P_7$, this is outcome (C).  
If $z \in V(P_1)$ between $v$ and $u$ (note that this is symmetric with $z \in V(P_3)$), then take 
$v' = y, u = t, s = s, t = r, x = x, P_1' = yP_5tP_6wP_1a, P_2' = yP_2b, P_3' = yP_2vP_3c, 
P_4' = tP_5sP_4x, P_5' = sP_4uP_1zP_7r$ to find outcome (B).  
If $z$ is on $P_1$ between $u$ and $w$, take $v' = y, u = t, s = s, t = u, x = x, P_1' = yP_5tP_6wP_1a, P_2' = yP_2b, P_3' = yP_3rP_7zP_1vP_3c, P_4' = tP_5sP_4x, P_5' = sP_4u$, which is outcome (B).  
If $z$ is on $P_1$ between $w$ and $a$, take 
$v' = y, u = r, t = u, s = v, x = x, P_1' = yP_2rP_7zP_1a, P_2' = yP_2b, P_3' = yP_5tP_6wP_1vP_3c, P_4' = rP_2a, P_5' = uP_4x$,
which gives outcome (C).  
If $z$ is on $P_2$ between $y$ and $b$, then take 
$v = v', P_1' = P_1, P_2' = vP_2rP_7zP_2b, P_3' = P_3, u = u, s = s, x = x, t = r, P_4' = P_4, P_5' = sP_5yP_2r$ 
to again find outcome (B).  
If $z$ is on $P_4$ between $u$ and $s$, take $v' = y, u = t, s = s, t = r, x = x, P_1' = yP_5tP_6wP_1a, P_2' = yP_2b, 
P_3' = yP_2vP_3c, P_4'= tP_5sP_4x, P_5' = sP_4zP_7r$ which is outcome (B).  
If $z$ is on $P_4$ between $s$ and $x$, take $v' = v, u = u, s = y, t = r, P_1' = P_1, P_2' = P_2, P_3' = P_3, 
P_4' = uP_4sP_5y, P_5' = rP_7zP_4x$ to get outcome (C).  
If $z$ is on $P_5$ or $P_6$, then take $v = v', P_1' = P_1, P_2' = P_2, P_3' = P_3, u = u, s = s, x = x, t = r, P_4' = P_4$.  
If $z$ is on $P_5$, then take $P_5 = sP_5zP_7r$ to find outcome (B) and if $z$ is on 
$P_6$, take $P_5 = sP_5tP_6zP_7r$ to find outcome (B). 
This completes the case when induction  yields outcome (B).

Next we assume that induction yields outcome (C).
Thus there exist vertices $s\in A\cap V(P_4)$, $t\in B\cup V(P_1)$, 
$w\in A\cap V(P_1)$ and $y\in X'$, and paths $P_5,P_6$ such that
$v,u,w,t,a$ occur on $P_1$ in the order listed, $P_5$ has ends $s$ and $t$,
$P_6$ has ends $w$ and $y$, and $P_5,P_6$ are disjoint and disjoint from $P_1,P_2,P_3,P_4$,
except for their ends.
Note that if $y \in X \cup B$, this is exactly outcome (D), including the notation.  
Suppose first that $y \in X \cup A$.  Then take $v' = v, u = u, x = y, P_1' = vP_1uP_4sP_5tP_1a, P_2' = P_2, P_3' = P_3, P_4' = uP_1wP_6y$ to find outcome (A).  
So we may assume that $y \in V(P_2)$.  
Then take $v' = w, u = t, s = s, t = u, x = x, P_1' = wP_1a, 
P_2' = wP_6yP_2b, P_3' = wP_1vP_3c, P_4' = tP_5sP_4x, P_5' = sP_4u$, which is outcome (C).  
Since $y \in V(P_3)$ is symmetric with this case, that completes this outcome.

Finally, we assume that induction yields outcome (D).
Thus there exist vertices $s,t\in A\cap V(P_4)$, $r\in B\cap V(P_4)$, $w\in B$ and $y,z\in A\cap X'$,
and paths $P_5,P_6,P_7$ such that $u,s,r,t,x$ occur on $P_4$ in the order listed,
$P_5$ has ends $s$ and $y$ and includes $w$, $P_6$ has ends $w$ and $t$, $P_7$ has ends $r$ and $z$,
and the paths $P_5,P_6,P_7$ are pairwise  disjoint and disjoint from $P_1,P_2,P_3,P_4$,
except for their ends.
Note that $y$ and $z$ are completely symmetric as are $r$ and $w$.  
Suppose that $y \in X$.  Then take $v' = v, u = u, x = y, P_1' = P_1, P_2' = P_2, P_3' = P_3, P_4' = uP_4sP_5y$ to get outcome (A).
So we may assume $y \in V(P_2)$.  
Suppose $z$ lies on $P_2$ between $v$ and $y$ (by symmetry, if $z$ lies on $P_2$, this assumption is without loss of generality).  
Then take $v' = v, u = u, t = r, x = x, s = s, P_1' = P_1, P_2' = vP_2zP_7rP_4sP_5yP_2b, P_3' = P_3, P_4' = uP_4s, 
P_5' = rP_4x$ which is outcome (C).  So $z$ lies on $P_3$.  
Then take $v' = y, u = u, s = s, t = r, x = x, P_1' = yP_2vP_1a, P_2' = yP_2b, P_3' = yP_5sP_4rP_7zP_3c, P_4' = uP_4s, P_5' = rP_4x$ which is again outcome (C).  
This completes this outcome and the proof.
\end{proof}

\section{Proof of Theorem \ref{MainThm}}
\label{sec:proof}


Let $H$ be a subgraph of a graph $G$.  By an {\em $H$-path} in $G$ we mean a path in $G$ with at least one edge, both ends in $V(H)$ and no other vertex or edge in $H$.

Let $H$ be a hex in a graph $G$, let
 $P$ be the union of a set of $H$-paths in $G$, and let
$Q$ be a subgraph of $H$.
We denote by $H + P - Q$ the graph obtained from $H\cup P$ by deleting all edges of $Q$ and then deleting all resulting
isolated vertices.
A typical application will be when $P$ and $Q$ are paths, but we will need more complicated choices.

A hex in a graph $G$ is \emph{optimal} if no hex of $G$ has strictly more odd segments.
We proceed by a series of lemmas, each improving 
a lower bound on the number of odd segments in an optimal hex.

\begin{lemma}
\label{lem:4odd}
Let $G$ be a $3$-connected bipartite graph.  Then every optimal hex of $G$ has at least four odd segments.
\end{lemma}
\begin{proof}
Let $(A, B)$ be a bipartition of $G$ and let $H$ be an optimal hex in $G$ with feet and segments numbered as in the definition of a hex.  We may assume for a contradiction that $H$ has at most three odd segments.  It follows that at least five feet of $H$ belong to the same set $A$ or $B$, and so we may assume that $v_1, v_2, \ldots, v_5$ all belong to $A$.  Thus $P_{14}$ has an internal vertex $u$ that belongs to $B$.  Since $G$ is $3$-connected we may assume, by replacing $P_{14}$ if necessary, that there exists an $H$-path $Q$ with one end $u$ and the other end, say $w$, in $V(H)-V(P_{14})$. By symmetry, we may assume that $w$ belongs to $P_{15}, P_{16}, P_{24}, P_{25}$, or $P_{26}$.  Let $R$ be defined as $v_1P_{15}w$, $v_1P_{16}w$, $v_4P_{24}w$, $P_{24}$ or $P_{24}$, respectively.  Then $H+Q-R$ is a hex with strictly more odd segments than $H$, contrary to the optimality of $H$.
\end{proof}

\begin{lemma}
\label{lem:5odd}
Let $G$ be an internally $4$-connected bipartite graph.  Then every optimal hex of $G$ has at least five odd segments.
\end{lemma}
\begin{proof}
Let $(A, B)$ be a bipartition of $G$ and let $H$ be an optimal hex in $G$ with feet and segments numbered as in the definition of a hex. 
By Lemma~\ref{lem:4odd} we may assume for a contradiction that $H$ has exactly four odd segments.  
It follows that two feet of $H$ in $\{v_1, v_2, v_3\}$ and two feet of $H$ in $\{v_4, v_5, v_6\}$ belong to the same set $A$ or $B$, and so we may assume that $v_1, v_2, v_4, v_5$ all belong to $A$.  Thus $P_{14}$ has an internal vertex $u$ that belongs to $B$.  Since $G$ is $3$-connected we may assume, by replacing $P_{14}$ if necessary, that there exists an $H$-path $Q$ with one end $u$ and the other end, say $w$, in $V(H)-V(P_{14})$.  
By symmetry, we may assume that $w$ belongs to $P_{15}, P_{16}, P_{25} P_{26}$, or $P_{36}$.
If $w$ belongs to $P_{15}, P_{16}$, $P_{36}$,  $A\cap V(P_{25}\cup P_{26})$, or $B\cap V(P_{26})$,
let $R$ be defined as $v_1P_{15}w, v_1P_{15}w$,  $P_{24}$, $P_{24}$, or $P_{16}$, respectively.  
Then $H+Q-R$ is a hex with strictly more odd segments than $H$, contrary to the optimality of $H$.

So we may assume that $w \in B\cap V(P_{25})$.  
Note that this is the last case and is not symmetric with anything else, so 
it suffices to reduce to any of the previous cases.
We now apply Lemma \ref{mainlemma} to the paths $v_1P_{14}u, uP_{14}v_4, Q$ and $X:=V(H) - (V(P_{14}) - \{w\})$ 
with $A$ and $B$ swapped. 
We may assume that the replacement paths do not change.  
If outcome (1) of the lemma holds, then there exist
vertices $y \in A\cap V(Q)$, $z \in X$ and an $H\cup Q$-path $R$ from $y$ to $z$.  
If $z \in B\cap V(P_{25})$, we may assume it belongs to $v_5P_{25}w$.  
Then we can replace $P_{25}$ by $v_2P_{25}wQyRzP_{25}v_5$, $Q$ by $vQy$, and apply the case above 
where $w \in A\cap V(P_{25})$.  
If $z \notin B$ or $z$ is not on $P_{25}$, then we can replace $Q$ with $uQyRz$ and apply one of the previous cases.

So we may assume that the second outcome of the lemma holds.
Thus there exist  vertices $a \in A\cap V(Q)$, $b \in B\cap V(P_{14})$, $c \in A\cap V(P_{14})$, 
and $d \in X$ and disjoint  $H\cup Q$-paths $R$ between $a$ and $b$ and $S$ between $c$ and $d$.  
We may assume that $b$ belongs to $v_1P_{14}u$.  
Then we can replace $P_{14}$ by $v_1P_{14}bRaQuP_{14}v_4$ and $Q$ by $uP_{14}cRd$ 
which puts us in one of the previous cases unless $d \in B\cap V(P_{25})$.  
So we may assume $d \in B\cap V(P_{25})$, and that it belongs to $wP_{25}v_5$.  
Let $H'$ be the hex obtained from $H\cup S\cup R\cup Q$ by deleting
$V(P_{15}\cup v_5P_{25}d\cup P_{34}\cup P_{35}\cup P_{36})-\{v_1,d,v_4,v_6\}$.
Then  $H'$ has nine odd segments,  contrary to the optimality of $H$.
\end{proof}

\begin{lemma}
\label{lem:6odd}
Let $G$ be an internally $4$-connected bipartite graph.  Then every optimal hex of $G$ has at least six odd segments.
\end{lemma}
\begin{proof}
Let $(A, B)$ be a bipartition of $G$ and let $H$ be an optimal hex in $G$ with feet and segments 
numbered as in the definition of a hex.  
By Lemma~\ref{lem:5odd} we may assume $H$ has exactly five odd segments and that
$v_1, v_2, v_4 \in A$ and $v_3, v_5, v_6 \in B$.  
We now apply Lemma \ref{mainlemma} to the paths $P_{14}, P_{15}, P_{16}$ and 
set $X:=V(H) - V(P_{14} \cup P_{15} \cup P_{16})$.
We may assume that the replacement paths do not change.

Suppose first that outcome (2) of the lemma holds.  
Thus we may assume that there exist vertices $u \in B\cap V(P_{14})$,
$w \in A\cap V(P_{15})$, $y \in B\cap V(v_1P_{15}w)$ and $z\in X$, and 
disjoint $H$-paths $Q$ from $u$ to $w$ and $R$ from $y$ to $z$.  
Then by symmetry we may assume that $z$ belongs to one of $P_{24}, P_{25}, P_{34}, P_{35}$.  
Let $T$ be, respectively, $v_2P_{24}z\cup P_{25}\cup P_{26}$, 
$v_5P_{25}z\cup P_{26}$,
$P_{25}\cup v_3P_{34}z$, 
$P_{36} \cup v_5P_{35}z$.  Then the hexes $H + (Q\cup R) - T$ have
at least six odd segments which is a contradiction.

So we may assume that outcome (1) of the lemma holds.  
Thus there exists a vertex $u \in B\cap V(P_{14})$ and an $H$-path $Q$ from it to $w \in X$.  
Then by symmetry we may assume that $w$ belongs to one of $P_{24}, P_{25}, P_{34},$ and $P_{35}$.  
For $w$ on $P_{24}, P_{34}, P_{35}$ or $A\cap V(P_{25})$, let $R$ be, respectively, 
$v_4P_{24}w, P_{36}, v_4P_{34}w, P_{34}$.  
Then $H+Q-R$ is a hex with more odd segments than $H$ which contradicts the optimality of $H$.

So we may assume $w \in B\cap V(P_{25})$.  
We now apply Lemma \ref{mainlemma} to the paths $v_1P_{14}u, uP_{14}v_4, Q$ and set  $X:=V(H) - V(P_{14}) - \{w\}$.  
Suppose that outcome (1) of the lemma holds.
Thus there exist vertices $y \in A\cap V(Q)$ and $z \in X$ and an $H\cup Q$-path $R$ from $y$ to $z$.  
If $z \in B\cap V(P_{25})$, we may assume it belongs to $v_5P_{25}w$.  
Then we can replace $P_{25}$ by $v_2P_{25}wQyRzP_{25}v_5$, $Q$ by $uQy$, 
and apply the case above where $w \in A\cap V(P_{25})$.  
If $z \in B \cap V(P_{26})$, then the hex $H + (Q \cup R) - (v_5P_{25}w \cup v_2P_{26}z \cup P_{35})$ 
has nine odd segments which contradicts the optimality of $H$.  
If $z \notin B$ or $z$ is not on $P_{25}$ or $P_{26}$, then we can replace $Q$ with $uQyRz$ and apply one of the previous cases.

So we may assume that the second outcome of the lemma holds.
Thus there exist vertices $a \in A\cap V(Q)$, $b \in B\cap V(v_1P_{14}u)$, 
$c \in A\cap V(P_{14})$ and $d \in X$, and disjoint $H\cup Q$-paths $R$ between $a$ and $b$ and $S$ between $c$ and $d$.  
Then we can replace $P_{14}$ by $v_1P_{14}bRaQuP_{14}v_4$ and $Q$ by $uP_{14}cRd$
which puts us in one of the previous cases, unless $d \in B$ and $d$ is on $P_{25}$ or $P_{26}$.  
Let $F = S \cup Q \cup R$.  If $d$ is on $wP_{25}v_5$, let $J = P_{35} \cup P_{26} \cup v_5P_{25}d \cup P_{15}$;
if $d$ is on $wP_{25}v_2$, let $J=P_{15}\cup P_{24}\cup P_{26}\cup v_2P_{25}d$;
 and if $d$ is on $P_{26}$, let $J = P_{24} \cup P_{35} \cup v_6P_{26}d$.  
Then the hexes $H + F - J$ have at least six odd segments, which contradicts the optimality of $H$.
\end{proof}

\begin{lemma}
\label{lem:allodd}
Let $G$ be an internally $4$-connected bipartite graph.  Then in every optimal hex of $G$ every segment is odd.
\end{lemma}
\begin{proof}
Let $(A, B)$ be a bipartition of $G$ and let $H$ be an optimal hex in $G$ with feet and segments 
numbered as in the definition of a hex.  
By Lemma~\ref{lem:6odd} we may assume $H$ has exactly six odd segments and that that 
$v_1, v_2, v_3, v_4 \in A$ and $v_5, v_6 \in B$.  
We now apply Lemma \ref{mainlemma} to the  paths $P_{14}, P_{15}, P_{16}$ and 
set $X:=V(H) - V(P_{14} \cup P_{15} \cup P_{16})$.

Suppose first that outcome (2) of the lemma holds.
Thus there exist vertices  $u \in B\cap V(P_{14})$,
$w \in A\cap V(P_{15})$,
 $y \in B$ and $z\in X$, and 
disjoint $H$-paths $Q$ from $u$ to $w$ and $R$ from $y$ to $z$.  
By  symmetry we may assume that $z$ belongs to $P_{24}\backslash v_2, P_{25}, or P_{26}$.  
Let $J$ be, respectively, $P_{35} \cup zP_{24}v_2$, $P_{34} \cup zP_{25}v_5$, 
or $P_{34} \cup zP_{26}v_6$. 
Then the hexes $H + (Q\cup R) - J$ each have nine odd segments, which contradicts the optimality of $H$.

So we may assume that outcome (1) one of the lemma holds.
Thus there exist  $u \in B\cap V(P_{14})$ and an $H$-path $Q$ from it to $w \in X$.  
Then we may assume that $w$ belongs to $P_{24}$ or $P_{25}$.  
If $w \in V(P_{24})$, let $R = v_4P_{24}w$.  
If $w \in A\cap V(P_{25})$, let $R = P_{24}$.  
Then $H + Q - R$ is a hex with nine odd segments, a contradiction.  
Thus it remains to handle the case when $w\in B\cap V(P_{25}\cup P_{26}\cup P_{35}\cup P_{36})$.

We now forget $w, Q, u$ and instead apply Lemma \ref{lemma2} to the paths $P_{14}, P_{15}, P_{16}$ 
and set $X:=V(H) - V(P_{14} \cup P_{15} \cup P_{16})$.
Outcomes (A) and (C) give results already ruled out by the case analysis from applying Lemma \ref{mainlemma}, 
so we may assume that  outcomes (B) or (D) hold.

Suppose that outcome (B) holds.  
Thus there exist vertices  $u \in B\cap V(P_{14})$, $w \in B \cap X$, $t \in A$ and $s \in B$,  
and a $H$-path $Q = u ... t ... w$ and a $H\cup Q$-path $R = t ... s$.  
By the previous analysis and symmetry we may assume that $w \in V(P_{25})$, so we are interested in where $s$ lies.  
The above case analysis handles the cases where $s$ is on $P_{24}$ or $P_{34}$, so we need only worry about 
the case where $s$ belongs to $P_{15}$, $wP_{25}v_5$, $P_{35}$, $P_{16}, P_{26}, P_{36}$.  
Then, respectively, let $J$ be defined as $P_{24} \cup wP_{25}v_5$, $P_{24} \cup wP_{25}s$, 
$P_{24} \cup wP_{25}v_5$, $P_{34} \cup P_{35} \cup P_{36}$, $P_{34} \cup P_{35} \cup P_{36}$, 
$P_{34} \cup P_{35} \cup v_3P_{36}s$.  
Then  $H + (Q\cup R) - J$ is a hex with nine odd segments, which contradicts the optimality of $H$.

So we must have outcome (D).  
Thus there exist vertices $u, w \in V(P_{14}) \cap B$, $r \in V(P_{14}) \cap A$, $s \in A$ and  $x, y \in X \cap B$, 
such that $v_1,w,r,u,v_4$ occur on $P_{14}$ in the order listed, and there exists a $H$-path
$Q = u ... s ... x$ and disjoint $H\cup Q$-paths $R = w ... s$ and $S = r ... y$.
Without loss of generality, we may assume that $x \in V(P_{25})$.  
By symmetry and taking advantage of the previous cases, we may assume that $y$
belongs to $xP_{25}v_5, P_{26}, P_{35}$, or $P_{36}$.  
Let $J$ be $P_{15} \cup P_{34} \cup P_{35} \cup P_{36} \cup v_5P_{25}y$, 
$P_{24} \cup xP_{25}v_5 \cup P_{36}$, 
$P_{15} \cup v_3P_{35}y \cup P_{34} \cup P_{36}$, 
or $P_{24} \cup xP_{25}v_5 \cup v_3P_{36}y$, respectively.  
Then $H + (Q\cup R\cup S) - J$ are each hexes with nine odd segments, a contradiction.
\end{proof}

\begin{proof}[Proof of Theorem~\ref{MainThm}]
Let $G$ be an internally $4$-connected non-planar bipartite graph.
By Theorem~\ref{kuratowskivar} the graph $G$ has a hex, and hence it has an optimal hex $H$.
By Lemma~\ref{lem:allodd} every segment   of $H$ is odd, as desired.
\end{proof}

\baselineskip 11pt
\vfill
\noindent
This material is based upon work supported by the National Science Foundation.
Any opinions, findings, and conclusions or
recommendations expressed in this material are those of the authors and do
not necessarily reflect the views of the National Science Foundation.

\end{document}